\documentclass{article}
\usepackage[utf8]{inputenc}
\usepackage[T1]{fontenc}

\title{Voter model stability with respect to conservative noises}
\author{
  Gideon Amir
  \thanks{Bar Ilan University, gideon.amir@biu.ac.il }
  \and
  Omer Angel
  \thanks{University of British Columbia, angel@math.ubc.ca }
  \and
  Rangel Baldasso
  \thanks{Pontifícia Universidade Católica do Rio de Janeiro, rangel@puc-rio.br }
  \and
  Daniel de la Riva
  \thanks{University of British Columbia, delariva@math.ubc.ca }
}

\date{\vspace{-5ex}}
\usepackage{amsmath, mathrsfs}
\usepackage{hyperref}
\usepackage{bm}
\usepackage{color}
\usepackage{amsthm}
\usepackage{bbm}
\usepackage{comment}
\usepackage{enumerate}
\usepackage[maxbibnames=99]{biblatex}
\usepackage{esvect}
\usepackage{graphicx}
\usepackage{amsmath,amssymb,amsfonts, upgreek}
\usepackage{tikz}
\usetikzlibrary{hobby}

\numberwithin{equation}{section}

\newtheorem{theorem}{Theorem}
\numberwithin{theorem}{section}

\newtheorem{definition}[theorem]{Definition}
\newtheorem{remark}[theorem]{Remark}
\newtheorem{lemma}[theorem]{Lemma}
\newtheorem{corollary}[theorem]{Corollary}

\newtheorem{conjecture}[theorem]{Conjecture}

\def\P{{\mathbb{P}}}

\def\Cov{\textup{Cov}}

\def\eps{\varepsilon}
\def\E{{\mathbb{E}}}
\def\F{\mathcal{F}}
\def\R{{\mathbb{R}}}
\def\Z{{\mathbb{Z}}}
\def\N{{\mathbb{N}}}

\newcounter{constant}
\newcommand{\nc}[1]{\refstepcounter{constant}\label{#1}}
\newcommand{\uc}[1]{c_{\textnormal{\tiny \ref{#1}}}}
\DeclareMathOperator{\ber}{Bernoulli}

\begin{document}
\maketitle
\begin{abstract}
    The notions of noise sensitivity and stability were recently extended for the voter model. In this model, the vertices of a graph have opinions that are updated by uniformly selecting edges. We further extend stability results to different classes of perturbations. We consider two different types of noise: in the first one, an exclusion process is performed on the edge selections, while in the second, independent Brownian motions are applied to such a sequence. In both cases, we prove stability of the consensus opinion provided the noise is run for a short amount of time, depending on the underlying graph structure. This is done by analyzing the expected size of the pivotal set, whose definition differs from the usual one in order to reflect the change associated with these noises.
\end{abstract}

\section{Introduction}

The notion of noise stability refers to the phenomenon in which a sequence of Boolean functions is resistant to small random perturbations of its input. This was originally introduced in the seminal paper by Benjamini, Kalai, and Schramm \cite{BKS99} together with its counterpart, called noise sensitivity. Initial analysis focused mainly on crossing functions for critical Bernoulli percolation and considered the case where noise acts independently in each coordinate of the input, effectively changing a positive fraction of its entries. Since then, this was generalized in many different directions, including Bernoulli percolation under exclusion noise \cite{exclusion, montonicity_exclusion, gv}, Boolean percolation \cite{ahlberg2014noise}, Voronoi percolation \cite{ahlbergb, QuenchedVoronoi, ahlberg2021rate} Bargmann-Fock percolation \cite{bargmann_fock}, and more recently Ising model under Glauber dynamics \cite{tassion2025noisesensitivitycrossingshigh}.

These notions were also studied in the context of social choice, for instance, in the case of majority voting \cite{odonnell_stablest}. More recently, the consensus opinion of the voter model was considered in \cite{Amir_2022}. This differs from previous works by considering a sequence of functions obtained by analyzing the evolution of a Markov process and introducing a notion of ``dynamical noise,'' where perturbations are applied in the evolution of the process.

Sensitivity questions regarding observables related to interacting particle systems are a relatively new area of investigation. Particularly challenging here is the fact that usually, these processes are constructed by means of auxiliary random variables that play very different roles and, thus, are not directly comparable. Sometimes, even the notion of noise to be considered and how it should affect each distinct source of randomness is not entirely clear. We begin investigating the case where noise is introduced in the graphical construction by conservative modifications of the ``clocks.'' We examine the consensus opinion of the voter model on finite graphs under the effect of noises that can be viewed as random permutations of the edge selections that determine the outcome of the function, as described below.

In order to precisely state our main result, let us first recall the voter model. Towards the noise sensitivity issue, we consider a sequence of graphs from the outset.
Fix a growing sequence of finite connected graphs $G_{n} = (V_{n}, E_{n})$ and a parameter $p \in (0,1)$. From here on, whenever we consider a sequence of graphs $(G_n)_{n \in \N}$, we will always assume $G_n$ has $n$ vertices. Let $\eta_{0} \in \{0,1\}^{V_{n}}$ denote the initial opinion of the voter model, whose entries are independently sampled with distribution $\ber(p)$. Let also $\omega$ denote an independent sequence of uniformly chosen directed edges in $E_{n}$. 
The voter model is now defined in terms of the initial configuration $\eta_{0}$ and the sequence $\omega$ as follows:
Given the configuration at time $k$, $\eta_{k}$, and $\omega(k+1) = (x,y)$, we let voter $y$ copy the opinion of voter $x$, i.e.
\[
\eta_{k+1}(z) = \begin{cases} \eta_k(z), & z\neq y \\
\eta_k(x), & z=y. \end{cases}
\]

It is well known that for any connected graph $G_{n}$, consensus is reached almost surely, i.e., all opinions must eventually agree. We denote this (random) consensus opinion by $f_{n}(\eta_{0}, \omega)$.  Since our results will not depend on the value of $p \in (0,1)$, we omit it from the notation. A formal definition can be found in Equation \eqref{eq:voter_step}.

We now introduce the noise we are interested in. 
We will change not the initial opinions, but the sequence $\omega$, reflecting a perturbation in the dynamics of the voter model.
Given $t>0$, let $\omega_{t}$ denote the sequence of edge selections obtained by performing an interchange process on the initial sequence $\omega$ up to time $t$. This amounts to performing a Markov process in the sequence $\omega$ that switches the order of any two neighboring entries at rate one (see Equation \eqref{eq:interchange_generator} for a precise definition). 

We now state the first main result of this paper, which will be proven in Section \ref{stabsec}:
\begin{theorem}\label{stabtheorem}
For any connected graph $G_{n}$ with $n$ vertices, maximum degree $\Delta_{n}$, and $t>0$, 
\begin{equation}\label{eq:stab_theorem}
\P \Big( f_{n}(\eta_{0},\omega) \neq f_{n}(\eta_{0},\omega_{t}) \Big) \leq 3 t \cdot \frac{ n \Delta_{n}}{
|E_{n}|} .
\end{equation}
\end{theorem}

\begin{remark}
\normalfont
    In the above theorem , $\eta_{0}$ can be taken to be any arbitrary initial configuration (random or deterministic) as long as it is independent of the voter model dynamics. However, since in the sensitivity regime, it is important that the entries are taken to be i.i.d. $\text{Bernoulli}(p)$ for $p\in (0,1),$ we decided to assume this condition throughout the rest of the paper.
\end{remark}

Inspired by Definition 1.5 in \cite{exclusion}, we provide the following.
\begin{definition}
 Let $(G_{n})_{n \in \N}$ be a sequence of connected graphs, and represent by $f_{n}(\eta_{0},\omega)$ the consensus opinion for the voter model on $G_{n}.$ We say that $\big(f_{n}(\eta_{0},\omega)\big)_{n \in \N}$ is \textbf{dynamically exclusion stable} if for any $\delta>0$, there exists $t>0$ such that for every n,
\begin{equation}
 \P\Big( f_{n}(\eta_{0},\omega) \neq f_{n}(\eta_{0},\omega_{t}) \Big)\leq \delta.
\end{equation}
\end{definition}

\begin{remark}
    \normalfont It is worth noting that the perturbation considered above acts on the voter model dynamics, hence the name, and is not an exclusion process per se, but rather an interchange process. We also remark that $f_{n}(\eta_{0},\omega)$ is a Boolean function on a generalized domain. Similar objects  were previously considered in many different contexts, for instance, functions on finite Abelian groups; see Chapter 8 in \cite{odonnell_2014} for a thorough presentation on the topic.  
\end{remark}

As an immediate corollary of Theorem \ref{stabtheorem}, we get the following result:
\begin{corollary}\label{cor:exclusion}
The consensus opinion of the voter model is dynamically exclusion stable for any sequence of connected graphs $(G_{n})_{n \in \N}$ such that $\sup \big( n \Delta_{n}/|E_{n}| \big) < \infty$. Furthermore, for any growing sequence of connected graphs and any sequence $t_n$, 
\begin{equation}
    \lim_{n \rightarrow \infty} \frac{t_{n} n \Delta_{n} } {|E_{n}|} = 0  \quad \Longrightarrow  \quad  \lim_{n \to \infty} \E \big[ f_{n}(\eta_{0},\omega)\cdot f_{n}(\eta_{0},\omega_{t_n})\big] =  p.
\end{equation}
\end{corollary}

In particular, the corollary above implies that the consensus opinion of the voter model is dynamically exclusion stable for any sequence of connected graphs with uniformly bounded degrees and for any sequence of regular graphs. The same also holds for the Erd\H{o}s-R\'enyi random graph $G(n,q)$ for $q$ above the connectivity threshold, since, for such $q$, the average degree is of the same order as the maximum degree with high probability (see Corollary $3.14$ of \cite{Bollobás_2001}).

In general, the usual notion of exclusion sensitivity implies standard noise sensitivity, see \cite{exclusion}. These two notions are not equivalent in general. In the dynamical setting, this is also the case, as \cite{Amir_2022} proves that the consensus opinion is dynamically noise sensitive for any sequence of graphs, and our result establishes dynamical exclusion stability for certain classes of graphs.

Additionally, we have the following notion of sensitivity that was also inspired by \cite{exclusion}.
\begin{definition}
    We say that the sequence of functions $f_{n}(\eta_{0}, \omega)$ is \textbf{dynamically exclusion sensitive} if for any $t>0$
    \begin{equation}
        \lim_{n \to \infty} \Cov \big(f_{n}(\eta_{0},\omega),f_{n}(\eta_{0},\omega_{t})\big) = 0.
    \end{equation}
\end{definition}

Attaining sensitivity is, in general, a harder problem, and as of now, we know little about it. However, we conjecture that the bound on Theorem \ref{stabtheorem} is the correct threshold, and for a sequence of graphs with  $\sup \big( n \Delta_{n}/|E_{n}| \big) = \infty$, the consensus opinion should be dynamically exclusion sensitive. Moreover, we believe that the following quantitative description is correct (see \cite{gps10,SS10,tassion2025noisesensitivitycrossingshigh,tassion_vanneuville} for a more recent exposure).
\begin{conjecture}\label{exc:conjecture}
Let $t_{n}$ be a non-negative sequence of times. Then 
\begin{equation}
    \lim_{n \rightarrow \infty} \frac{t_{n} n \Delta_{n} } {|E_{n}|} = 0  \quad \Longrightarrow  \quad \lim_{n \to \infty} \E \big[ f_{n}(\eta_{0},\omega)\cdot f_{n}(\eta_{0},\omega_{t_n})\big] =  p, 
\end{equation}
\begin{equation}
    \lim_{n \rightarrow \infty} \frac{t_{n} n \Delta_{n} } {|E_{n}|} = +\infty \quad \Longrightarrow \quad  \lim_{n \to \infty} \Cov \big(f_{n}(\eta_{0},\omega),f_{n}(\eta_{0},\omega_{t_n})\big) =0,
\end{equation}
 and for $ \lim_{n\rightarrow \infty} t_{n} n \Delta_{n}/|E_{n}| \in(0,\infty)$ it is neither sensitive nor stable. 
\end{conjecture}

\begin{remark} 
\normalfont
    A natural question that arises is what happens when an exclusion process is performed on the initial opinions instead of the components that determine the dynamics. The consensus opinion reached by the voter model can be seen as a random dictator function. The dynamics chooses which of the input bits will dictate the consensus opinion and the initial opinion of this selected vertex will become the consensus over the whole graph. It can be shown that the dictator is uniformly distributed among the vertices of the graph, regardless of the graph structure (this follows, e.g., from the duality of the voter model with coalescing random walks as in~\eqref{eq:duality_consensus}).
    Since the dynamics are not perturbed, the random dictator is the same for both the original and the noised initial configurations. Given this random dictator, the consensus opinion of the noisy initial configuration will be the same as in the original one if the evolution of the exclusion process keeps the same particle in the position of the random dictator, and will be given by an independent Bernoulli trial otherwise. In particular, the stability behavior can be described by analyzing the probability that a random walk starting from a uniformly selected vertex of $G_n$ is at this starting vertex at time $t_n$. In the complete graph case, if $t_n n \to 0$, one gets stability, while sensitivity of the consensus opinion holds in the regime when $t_n n$ diverges. On the other hand, in the star graph with $n$ vertices, one detects stability when $t_n \to 0$ and sensitivity when $t_n$ diverges. In this case, the structure of the graph plays an important role, and the two examples above imply that the maximum degree is not the correct observable for this threshold. Conservative noises highlight the local geometry of the graph, adding a layer of complexity when compared to the resampling dynamics analyzed in~\cite{Amir_2022}. This also seems to be reflected in the dynamical exclusion as in Conjecture~\ref{exc:conjecture}.
\end{remark}

Although the continuous-time counterpart of the voter model is more often considered, our observable of interest is not affected by the fact that we work with the discrete-time version of the process. Furthermore, the type of noise considered above also does not rely on the temporal distance between the relevant events for the dynamics; hence, the choice to work with the discrete-time version. 

However, there are several dynamical versions of the Poisson process, and so different types of perturbation can be considered. For example, if each edge activation is removed at rate 1, and new edge activations are added at rate $1$ in space-time, then the stationary state is the Poisson process. This is the perturbation that was studied in \cite{Amir_2022}. One more, which we also consider in the present work, is a continuous version of $\omega_t$, where the times at which edges are activated evolve as independent Brownian motions. For such a noise, the spatial distance between edge activations is actually taken into account, and so, in order to introduce it, some additional notation is needed.

In order to construct the continuous-time voter model $(\xi_{t})_{t \geq 0}$, in addition to the initial condition $\xi_{0}:=\eta_{0}$, one needs to keep track of a Poisson point process that controls the clock rings and edge selections. Denote by $\theta$ a Poisson point process on $\vec{E}_{n} \times \R_{+}$ with intensity measure $\mu \otimes \lambda$, where $\mu$ denotes the counting measure on the set of oriented edges of the graph $G_{n}$ and $\lambda$ is the Lebesgue measure on $\R_{+}$. With these two ingredients at hand, the continuous-time voter model can be constructed as in Section \ref{prelim}. Let $f_{n}(\xi_{0}, \theta)$ denote the consensus opinion of the continuous-time voter model in $G_{n}$ generated by the pair $(\xi_{0}, \theta)$, as in the case of $f_{n}(\eta_{0},\omega)$

For the perturbation, denote by $\theta_{s}$ the noisy version of $\theta$ where each time-stamp is moved according to an independent Brownian motion up to time $s$ (reflected at the origin). We now state our main result regarding the Brownian noise.
\nc{c_brownian_1}
\nc{c_brownian_2}
\begin{theorem}\label{t_brownian_noise}
    Given $\uc{c_brownian_1}>0$, there exists $\uc{c_brownian_2}>0$ such that the following holds. For any connected graph $G_{n} = (E_n, V_n)$ on $n$ vertices and $s>0$ such that $|E_{n}|\sqrt{s} \leq \uc{c_brownian_1}$,
    \begin{equation}
         \P \Big( f_{n}(\xi_{0},\theta) \neq f_{n}(\xi_{0},\theta_{s}) \Big) \leq \uc{c_brownian_2} n \Delta_{n} \sqrt{s}.
    \end{equation}
\end{theorem}

Similarly to Corollary \ref{cor:exclusion}, we have the following.
\begin{corollary}
Let $(G_{n})_{n\in \N}$ be a growing sequence of finite connected graphs such that $\sqrt{s_{n}}\cdot n \cdot \Delta_{n} \rightarrow 0$. For any sequence $s_{n}$ such that $n \cdot s_{n} \rightarrow 0$,
\begin{equation}
    \liminf_{n \rightarrow \infty} \P\Big( f_{n}(\xi_{0},\theta) \neq f_{n}(\xi_{0},\theta_{s_{n}}) \Big) = 0.
\end{equation}
\end{corollary}

Let us now briefly discuss the condition $|E_{n}| \sqrt{s} \leq \uc{c_brownian_1}$ in the statement of Theorem \ref{t_brownian_noise}. This can be thought of as a ``normalization'' factor. Indeed, the distance between two consecutive edge rings is on average $\frac{1}{2|E_{n}|}$. So, by taking a sequence of times such that $|E_{n}| \sqrt{s_{n}} \rightarrow \infty,$ one should expect too many time stamps to be exchanged, incurring many pivotal swaps (see Definition \ref{def_pivotality}). Naturally, one should not expect stability in such a case. Of course, it remains to determine whether the condition $n \Delta_{n} \sqrt{s_{n}} \rightarrow 0$ is necessary or if $|E_{n}| \sqrt{s_{n}} \rightarrow 0$ is enough. In fact, this boils down to answering an analogous question regarding Theorem \ref{stabtheorem}, as the bound we are proving is
\begin{equation}
         \P \Big( f_{n}(\xi_{0},\theta) \neq f_{n}(\xi_{0},\theta_{s}) \Big) \leq \uc{c_brownian_2}  \sqrt{s} |E_{n}| \bigg(\frac{ n \Delta_{n}}{
|E_{n}|} + o(1)\bigg).
    \end{equation}
The term $n \Delta_{n} /|E_{n}|$ arises from the proof of Theorem \ref{stabtheorem}. A proof of Conjecture \ref{exc:conjecture} would hint toward the fact that the same should hold for the Brownian case.

\bigskip

\noindent\textbf{Proof overview.} The proof of many results regarding noise stability and sensitivity relies on the analysis of pivotality of the entries of the Boolean function under consideration. We employ the same strategy, but the role of pivotal entries is replaced by \textrm{pivotal flips}, that is, instances where exchanging the position of two consecutive edge selections yields different final consensus opinions in the model. The precise definition of pivotal flip is given in Definition \ref{def_pivotality}.

With this notion in hand, we proceed to bound the number of such flips by means of characterizing the different ways such a phenomenon can occur. For Theorem \ref{stabtheorem}, the bound in~\eqref{eq:stab_theorem} is obtained by counting the number of pivotal flips that occur up to some time $t$ and the simple observation that, in order for the consensus opinions to differ in the original and perturbed configurations, at least one pivotal flip must occur before time $t$. This is the main result of Section \ref{stabsec}.

The strategy used in Theorem \ref{stabtheorem} also works as a proxy for Theorem \ref{t_brownian_noise}. Indeed, ignoring the geometry induced by the continuous time, one can imagine that the flips in the order of time stamps due to the Brownian motions could be compared to the case where the noise is given by an exclusion process. Further strengthening this intuition, the consensus opinion is only affected by the Brownian perturbation when it causes a change in the order of the edge selections. What then remains to be analyzed is the probability that the Brownian noise is strong enough to incur such swaps. This is done in detail in Section~\ref{sec: Brownian}.

\section{Preliminaries}\label{prelim}

\noindent \textbf{The voter model.}
In this paper, we consider both the discrete and continuous-time voter models, which we now define on a growing sequence of finite connected graphs $G_{n} = (V_{n}, E_{n})$, $n \in \N$. We assume that $|V_{n}| = n$, for all $n \in \N$, and write $\vec{E}_{n}$ for the \emph{directed} set of edges, i.e., for each edge $e=\{x,y\} \in E_{n} $, we add two directed edges $\vec{e}_{1} = (x,y)$ and $\vec
{e}_{2} = (y,x)$ to the set $\vec{E}_{n}$.

We first define the voter model in discrete time. In this case, the process $(\eta_{k})_{k \geq 0}$ relies on two different ingredients. First, sample  the initial configuration $\eta_{0} \in \{0,1\}^{V_{n}}$ having i.i.d.\ $\ber(p)$ entries, for some fixed $p \in (0,1)$. Second, the evolution depends on a sequence of random directed edges $\omega = \big(\omega(k)\big)_{k \geq 1}$, where each $\omega(k)$ is independent and uniformly sampled from $\vec{E}_{n}$. We further assume that $\eta_{0}$ is independent from $\omega$. The evolution is set inductively: assume $\eta_{k}$ and $\omega(k+1) = (x_{k+1}, y_{k+1})$ are known. At step $k+1$, vertex $y_{k+1}$ copies the opinion of $x_{k+1}$:
\begin{equation}\label{eq:voter_step}
    \eta_{k+1}(z) = \begin{cases}
        \eta_{k}(x_{k+1}), & \quad \text{if } z=y_{k+1}, \\
        \eta_{k}(z), & \quad \text{otherwise}.
    \end{cases}
\end{equation}

The definition above immediately implies that, conditioned on $\eta_{0}$, the voter model is a discrete-time Markov chain with finite state space and two absorbing states: the constant configurations $\underline{1}$ and $\underline{0}$. Standard Markov Chain theory then implies that the chain will almost surely eventually fixate in one of these two configurations, called the consensus state. Let us denote by $f_{n}(\eta_{0}, \omega)$ the random variable that indicates the consensus state the chain achieves in the graph $G_{n}$. We study $f_{n}$ as a sequence of functions of the initial configuration $\eta_{0},$ and the edge sequence $\omega$.

Our interest lies in analyzing sensitivity with respect to perturbations of the vector $\omega$ (see \cite{Amir_2022}). For this reason, we remove $\eta_{0}$ from the notation and introduce the following notion of pivotality.
\begin{definition}\label{def_pivotality}
    Given a configuration $\omega$ and $i,j \in \N$, let $\omega_{i\to j}$ denote the configuration obtained from $\omega$ by moving the $i$-th edge selected to position $j$. In particular, if $j > i$, then the edge sequence in $\omega_{i\to j}$ is given by
\begin{equation}
    \Big(\omega(1),\cdots, \omega(i-1),\omega(i+1),\cdots, \omega(j),\omega(i),\omega(j+1), \dots\Big),
\end{equation}
with a similar formula if $j<i$.
We say that the transition $i \to j$ is \textbf{pivotal} if
\begin{equation}
    f_{n}(\omega) \neq f_{n}(\omega_{i \to j}).
\end{equation}
\end{definition}

\bigskip

Let us now turn our attention to the continuous-time voter model. In order to avoid confusion with its discrete counterpart, we denote the continuous-time voter model by $\big( \xi_{t} \big)_{t \geq 0}$ In this case, we define the process through its generator: for any function $f:\{0,1\}^{V_{n}} \to \R$, define
\begin{equation}
    \mathcal{L} f(\xi) = \sum_{x \in V_{n}} \sum_{y \sim x} f(\xi^{x,y})-f(\xi),
\end{equation}
where $\xi^{x,y}$ is obtained from the configuration $\xi$ via
\begin{equation}
    \xi^{x,y}(z) =
        \begin{cases}
            \xi(x), \quad \text{if } z=y, \\
            \xi(z), \quad \text{otherwise}.
        \end{cases}
\end{equation}

We now describe a graphical construction of the continuous voter model, which we will use throughout the text. Consider a Poisson point process on the set $\vec{E}_{n} \times \R$ with intensity $\mu \otimes \lambda$, where $\mu$ denotes the counting measure in $\vec{E}_{n}$ and $\lambda$ denotes the Lebesgue measure in $\R$. Fix a realization $\theta:=(\omega(k), t(k))_{k \in \N}$ of the Poisson point process with $0 < t(1) < t(2) < \dots$. Let us briefly mention that, although the time stamps in the Poisson point process take values in the whole real line, the realizations above only take into account the sequence of times such that $t(k)>0$. From the realization $(\omega(k), t(k))_{k \in \N}$ and an arbitrary initial condition $\xi_{0}$, we define the process $\big( \xi_{t} \big)_{t \geq 0}$ by means of its discrete counterpart $(\eta_{k})_{k \in \N}$, constructed with the sequence of edges $\big( \omega(k) \big)_{k \in \N}$ and initial condition $\xi_{0}$. Given $(\eta_{k})_{k \in \N}$, set
\begin{equation}
    \xi_{t} = \xi_{0} \textbf{1}_{ [0,t(1)) }(t) + \sum_{k \geq 1} \eta_{k} \textbf{1}_{ [t(k), t(k+1)) }(t).
\end{equation}

It is clear from the construction that $\eta$ is the skeleton chain of $\xi$. In particular, the same considerations about the absorption of the process into one of the constant configurations hold. In this case, we will denote this consensus opinion as $f_{n}(\xi_{0}, \theta)$.

\begin{remark}
    Although the consensus opinion does not depend on the sequence of times $(t(k))_{k \geq 1}$, the notion of noise we introduce in Section~\ref{sec: Brownian} takes into account the spacing between two clock rings. This is the reason why we introduce both versions of the model here.
\end{remark}

\bigskip

\noindent \textbf{Duality and consensus.} The voter model has a  well-known duality relation with coalescing random walks that is key to many results regarding it (see \cite{Liggett1985InteractingParticleSystems}). We will focus our attention here on the discrete-time model, but analogous constructions hold in continuous time.

For each fixed $T \geq 1$ and $k \in V_{n}$, denote by $\big( X_{k}^{T}(t) \big)_{ t \in [0, T]}$ the random walk that is at position $k$ at time $T$ and runs backward in time using the edge selections from $\omega$ reversed. That is, if the random walk is at position $j$ at time $\ell+1$, it jumps to $i$ at time $\ell$ if $\omega(\ell+1) = \vv{ij}$. For large enough $T >0$, the collection of walks running backward in time coalesce.

To see the relation with the voter model, one should trace the origin of the opinion of vertex $k \in  V_{n}$ at time $T$. Going backward in time, the opinion jumps from $j$ to $i$ if the edge $(i,j)$ is selected, which is precisely the random walks defined above.

The key observation is that, provided $T$ is chosen large enough, all of these random walks are likely to coalesce by time $T$. Thus, the final opinion of the voter model is obtained by observing the opinion of the position of the remaining random walk at time zero. Denote by $\mathcal{C}(\theta)$ the minimal time $T$ such that the system of backward random walks from time $T$ coalesce before reaching time $0$, called the coalescence time. 

The observation above also implies that 
\begin{equation}\label{eq:duality_consensus}
f(\eta_{0},\omega) \sim \eta_0(X),
\end{equation}
where $X$ is distributed according to the invariant measure of the (lazy) random walk in $G_{n}$ with conductances.

\begin{equation}
    c_{x,y} = \begin{cases}
                  \frac{1}{2|E_{n}|}, & \quad \text{if } x \sim y, \vspace{1.5 mm} \\
                  \frac{2|E_{n}|-\deg(x)}{2|E_{n}|}, & \quad \text{if } x = y, \\
                  0, & \quad \text{otherwise}.
              \end{cases}
\end{equation} 

A simple calculation shows that the invariant distribution of the random walk above is uniform over the vertices of the graph.

Another important consequence that we use throughout the paper is that, since the invariant measure for the random walk above is uniform, the probability that the consensus function is decided by a chosen vertex at a certain time is $1/n$. That is, for every $\ell \geq 0$ and $k \in V_{n}$, let $\omega^{\ell}:=(\omega(\ell+1),...)$ and $\text{Dict}(\omega^{\ell})$ be the function of the edge selections whose output is the vertex from which all the other opinions stem from.  We have the following:

\begin{equation}\label{probdict}
\P\big(\text{Dict}(\omega^{\ell})= k \big) = \frac{1}{n}.
\end{equation}

\nc{c_coal}
In the proof of the lemma above, we will use the following bound on the coalescence time for independent random walks on a finite graph, due to Oliveira \cite{ROliveira}

\begin{theorem}[\cite{ROliveira}, Theorem 1.1]\label{t:coal}
There exists a universal positive constant $\uc{c_coal}$ such that the coalescence time $\mathcal{C}^{*}$ of continuous-time simple random walk in a finite connected graph $G$ satisfies
\begin{equation}
    \E[\mathcal{C}^{*}] \leq \uc{c_coal} T_{\textnormal{hit}},
\end{equation}
where $T_{\textnormal{hit}}$ is the maximum expected hitting time of $G$ across all possible starting positions and target vertices. 
\end{theorem}

Finally, Theorem 11.5 of~\cite{LevinPeresWilmer2006} (originally proven in~\cite{inproceedings}) yields
\begin{equation}\label{t:bound_thit}
    T_{\textnormal{hit}} \leq 2(n-1)|E_{n}|
\end{equation}
for the discrete-time simple random walk. Since the continuous time random walk jumps at rate at least 1, and thus the same bound holds in continuous time as well.

\subsection{The interchange process}

Let us now introduce the first noise operator, which will be applied to the discrete-time voter model. This noise will alter the order of the edge selections in the sequence $\omega$ via the so-called \textbf{interchange process} that we define now.

The interchange process in $\N$ is the interacting particle system with state space $\N^{\N}$, initial state $\zeta_{0}(k)=k$, for every $k \in \N$, and generator given by
\begin{equation}\label{eq:interchange_generator}
    \mathcal{L}{\mathfrak{f}}(\zeta) = \frac{1}{2}\sum_{k=1}^{\infty} \mathfrak{f} (\zeta^{k,k+1})-\mathfrak{f} (\zeta),
\end{equation}
where $\mathfrak{f}$ is any local function of the configurations and $\zeta^{k,k+1}$ is obtained from $\zeta$ via
\begin{equation}
    \zeta^{k,k+1}(z) = \begin{cases}
        \zeta(k+1), & \quad \text{if } z=k, \\
        \zeta(k), & \quad \text{if } z=k+1, \\
        \zeta(z), & \quad \text{otherwise}.
    \end{cases}
\end{equation}

We denote by $\zeta_{t}$ the state of the interchange process at time $t$. Equivalently, the interchange process is the process where each positive integer site $k$ starts at state $k$ and nearest-neighbor states are interchanged independently with rate one. From this definition, it follows that for any positive time $t$, $\zeta_{t}$ defines a permutation of $\N$.

We remark that the interchange process can be seen as a generalization of the exclusion process. In fact, given any initial condition $\Xi_{0} \in \{0,1\}^{\N}$ of the exclusion process, we can obtain the configuration at time $t$ as $\Xi_{t} = \Xi_{0} \circ \zeta_{t}$. We also remark that this process can be defined analogously for finite intervals of $\N$ (this will be convenient in later sections).

Given $t>0$, the noised configuration $\omega_{t} = \big( \omega_{t}(k) \big)_{k \geq 1}$ is obtained from $\omega$ (shortened notation of $\omega_{0}$) and $\zeta_{t}$ via
\begin{equation}
    \omega_{t}(k) = \omega\big(\zeta_{t}(k)\big).
\end{equation}

Observe that the sequence $\omega_{t}$ includes the same edges as $\omega$, but in a different order determined by $\zeta_{t}$. The larger $t$ is, the noisier is the configuration $\omega_{t}$ with respect to $\omega$, since more transpositions are applied to construct $\zeta_t$. Notice that if $\omega$ is a sequence of i.i.d.\ directed edges, then so is $\omega_{t}$, although $\omega$ and $\omega_{t}$ are not independent from each other.

\section{Exclusion Stability for the Voter model}\label{stabsec}

In this section we prove Theorem~\ref{stabtheorem}. The proof is based on three preliminary results. The following lemma states that a lazy random walk on $[0,N]$ starting at $k$ makes on average $k(N-k)$ non-zero steps before hitting $\{0,N\}$, even if the lazy parameter can change arbitrarily. 
\begin{lemma}\label{rwlemma}
Let $(p_{n})_{n \in \N}$ and $(X_{n})_{n \in \N}$ be two processes adapted to a given filtration $(\mathcal{F}_{n})_{n \in \N}$. Assume that these processes take values in $[0,1]$ and $\N$, respectively, that $X_{0} = k \in [0,N]$, and
\begin{equation}\label{randomwalkevolve}
\begin{split}
\P\big( X_{n+1} = X_{n} + 1 \big| & \mathcal{F}_{n} \big) = \P\big( X_{n+1} = X_{n} - 1 \big| \mathcal{F}_{n} \big) = \frac{p_{n}}{2} \\
\P\big( X_{n+1} = X_{n} \big| & \mathcal{F}_{n} \big) = 1-p_{n}.
\end{split}
\end{equation}
Let $\tau$ denote the hitting time of the set $\{0,N\}$ by the process $(X_{n})_{n \in \N}$. If $\E[\tau] < \infty$, then
\begin{equation}
    \E \bigg[ \sum_{i=0}^{\tau-1} p_{i} \bigg] = k(N-k).
\end{equation}
Furthermore, if there exists $\delta_{N} > 0$ such that $p_{n} \geq \delta_{N}$ whenever $X_{n} \in \{1,..., N-1\}$, then $\E[\tau]< \infty$.
\end{lemma}

\begin{proof}
The first part of the lemma is a consequence of the following martingale argument. Define the martingale $(M_{n})_{n \in \N}$ by setting $M_{0}:=k^{2}$, and define for $n\geq 1$
\begin{equation}
M_{n}:= X_{n}^{2} - \sum_{i=0}^{n-1}p_{i}.
\end{equation}
Since $E[\tau]< \infty,$ and $\E\big[ \left|M_{n+1}-M_{n}\right| \big| \mathcal{F}_{n}\big]$ is almost surely bounded on the event $\{\tau > n\}$, the Optional Stopping Theorem yields
\begin{equation}
\E[M_{\tau}] = \E[M_{0}] = k^{2}.
\end{equation}

Applying this fact, we get that
\begin{equation}
\frac{k}{N}\bigg( N^{2} - \E \bigg[ \sum_{i=0}^{\tau-1} p_{i} \bigg] \bigg) + \bigg(1-\frac{k}{N}\bigg) \bigg( 0 - \E\bigg[ \sum_{i=0}^{\tau-1} p_{i} \bigg] \bigg) = k^{2},
\end{equation}
implying the desired result.

For the second part, the uniform lower bound on $p_n$ implies that for every $j\geq0$, \linebreak $\P(\tau\leq(j+1)N \, | \, \tau > jN) \leq 1-\delta_N^N$. Therefore, $\P(\tau>jN)\leq \big(1-\delta_N^N)^j$, and thus $\E[\tau]<\infty$ and, in fact, $\tau/N$ are stochastically dominated by a geometric random variable.
\end{proof}

Consider the (random) collection of pivotal swaps for the configuration $\omega_{s}$, defined as
\begin{equation}
    \text{Piv}_{s}:= \{ \ell \in \N: \text{the transition } \ell \leftrightarrow \ell+1 \text{ is pivotal for } \omega_{s} \}
\end{equation}
The next lemma claims that the expected size of this set provides a bound on the probability that the noise alters the consensus opinion.
\begin{lemma}\label{l_pivotal_bound}
    We have
\begin{equation}\label{eq:piv_integral}
\P\big( f_{n}(\omega) \neq f_{n}(\omega_{t}) \big) \leq \int_{0}^{t} \E[|\textnormal{Piv}_{s}|]ds = t\E[|\textnormal{Piv}_{0}|].
\end{equation}
\end{lemma}

\begin{proof}
In order to verify the bound, denote by $\big( N_{t} \big)_{t \geq 0}$ the process that counts the number of pivotal flips. The key observation is that this is an inhomogeneous Poisson process with random adapted intensity rate given by $r_t = | \text{Piv}_{t}|$. In particular,
\begin{equation}
\P\big( f_{n}(\omega) \neq f_{n}(\omega_{t}) \big) \leq \P \big (N_{t} \geq 1 \big) \leq \E[ N_{t} ] = \int_{0}^{t} \E[|\text{Piv}_{s}|]ds.
\end{equation}

Since $\omega_{s} \sim \omega_{0}$, for all $s \geq 0$, the distribution of the set $\text{Piv}_{s}$ does not depend on $s$, and thus
\begin{equation}\label{eq:pivotal_integral}
\int_{0}^{t} \E[|\text{Piv}_{s}|]ds = t \E[|\text{Piv}_{0}|],
\end{equation}
concluding the proof of the lemma.
\end{proof}

Finally, the third auxiliary result gives a bound on the expected size of the pivotal set.
\begin{lemma}\label{lemma_expected_piv}
We have
\begin{equation}
     \E[|\textnormal{Piv}_{0}|] = \sum_{\ell=1}^{\infty}\P(f(\omega) \neq f(\omega_{\ell \leftrightarrow \ell+1})) \leq \frac{3n \Delta_{n}}{|E_{n}|}.
\end{equation}
\end{lemma}

\begin{proof}
In the following, we drop the subscript and simply write $\omega = \omega_{0}$ for the sequence of initial edge selections. If the transition $\ell \leftrightarrow \ell+1$ is pivotal for $\omega$, then the edges $\omega(\ell)$ and $\omega(\ell+1)$ necessarily share a common vertex, since otherwise the changes in opinion induced by them commute. 
This is also the case if $\omega(\ell),\omega(\ell+1)$ share a starting point.
Figure \ref{fig:edges} illustrates all the combinatorial possibilities for a pivotal event.

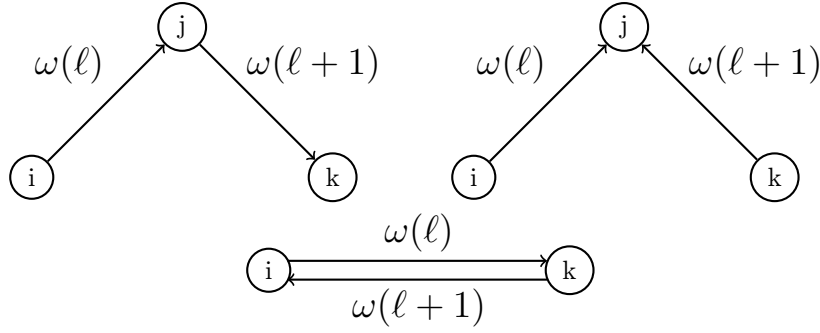
\begin{figure}
\begin{center}
\begin{tikzpicture}[node distance={15mm}, thick, main/.style = {draw, circle}] 
  \node at (2.5,9.5) {\Large $\omega(\ell)$};
   \node at (5.75,9.5) {\Large $\omega(\ell+1)$};
  \node[main] (n6) at (2,8) {i};
  \node[main] (n4) at (4,10)  {j};
  \node[main] (n5) at (6,8)  {k};
\draw[->] (n6) -- (n4); 
\draw[->] (n4) -- (n5); 
\end{tikzpicture}
\qquad 
\begin{tikzpicture}[node distance={15mm}, thick, main/.style = {draw, circle}] 
  \node at (2.5,9.5) {\Large $\omega(\ell)$};
   \node at (5.75,9.5) {\Large $\omega(\ell+1)$};
  \node[main] (n6) at (2,8) {i};
  \node[main] (n4) at (4,10)  {j};
  \node[main] (n5) at (6,8)  {k};
\draw[->] (n6) -- (n4); 
\draw[->] (n5) -- (n4); 
\end{tikzpicture}
\qquad
\begin{tikzpicture}[node distance={15mm}, thick, main/.style = {draw, circle}] 
  \node at (4,8.5) {\Large $\omega(\ell)$};
   \node at (4,7.5) {\Large $\omega(\ell+1)$};
  \node[main] (n6) at (2,8) {i};
  \node[main] (n5) at (6,8)  {k};
 \draw[->] (2.25, 8.125) -- (5.70, 8.125);
  \draw[->] (5.70, 7.875) -- (2.25, 7.875);
\end{tikzpicture}
\end{center}
\caption{The three ways in which two edges can overlap if the induced opinion changes do not commute.}
\label{fig:edges}
\end{figure}

We first deal with the left image and estimate its contribution to $\E[|\text{Piv}_{0}|]$. 
Here, the $\ell$-th edge $\omega(\ell)$ is the directed edge $\vv{ij},$ and $\omega(\ell+1)$, is the directed edge $\vv{jk}$. 
In order for this exchange to be pivotal, it is necessary and sufficient that vertices $i$ and $j$ have different opinions before the $\ell$-th edge selection, that is $\eta_{\ell-1}(i) \neq \eta_{\ell-1}(j)$, and that vertex $k$ is the dictator based on the subsequent dynamics after the $\omega(\ell+1)$ activity, i.e. the consensus opinion satisfies $\text{Dict}(\omega^{\ell+1})=k$. 
Hence, the contribution from such pairs of edges to the expected pivotal set is given by 
\begin{equation}\label{eq: comb}
\E\bigg[\sum_{\ell=1}^{\infty} \sum_{\vv{ij} \in \vv{E_{n}}} \sum_{\vv{jk} \in \vv{E_{n}}}\mathbbm{1}_{\omega(\ell)=\vv{ij}}\cdot\mathbbm{1}_{\omega(\ell+1)=\vv{jk}} \cdot\mathbbm{1}_{\eta_{\ell-1}(i) \neq \eta_{\ell-1}(j)}\cdot\mathbbm{1}_{\text{Dict}(\omega^{\ell+1})=k}\bigg].
\end{equation}
 Let us now consider the top right case, where both edges point toward the same vertex $j$. In this case, it is necessary and sufficient that vertex $j$ is the dictator from the $(\ell+1)$-th edge onward, and vertices $i$ and $k$ have disagreeing opinions. Hence, the contribution of this term to the expected value is given by
\begin{equation}
\E\bigg[\sum_{\ell=1}^{\infty} \sum_{\vv{ij} \in \vv{E_{n}}} \sum_{\vv{kj} \in \vv{E_{n}}}\mathbbm{1}_{\omega(\ell)=\vv{ij}}\cdot\mathbbm{1}_{\omega(\ell+1)=\vv{kj}} \cdot\mathbbm{1}_{\eta_{\ell-1}(i) \neq \eta_{\ell-1}(k)}\cdot\mathbbm{1}_{\text{Dict}(\omega^{\ell+1})=j}\bigg].
\end{equation}
Observe that if $\mathbbm{1}_{\eta_{\ell-1}(i) \neq \eta_{\ell-1}(k)}$ it must be the case that either one of $\mathbbm{1}_{\eta_{\ell-1}(i) \neq \eta_{\ell-1}(j)}$ or $\mathbbm{1}_{\eta_{\ell-1}(k) \neq \eta_{\ell-1}(j)}$ holds. 
By symmetry these have the same probability, and so the expectation above equals
\begin{equation}
2\cdot\E\bigg[\sum_{\ell=1}^{\infty} \sum_{\vv{ij} \in \vv{E_{n}}} \sum_{\vv{kj} \in \vv{E_{n}}}\mathbbm{1}_{\omega(\ell)=\vv{ij}}\cdot\mathbbm{1}_{\omega(\ell+1)=\vv{kj}} \cdot\mathbbm{1}_{\eta_{\ell-1}(i) \neq \eta_{\ell-1}(j)}\cdot\mathbbm{1}_{\text{Dict}(\omega^{\ell+1})=j}\bigg].
\end{equation}
Since an edge appears in each direction with the same probability, this is exactly twice the value  of \eqref{eq: comb}.
Finally, the bottom configuration can be seen as a special case of the first one, by allowing $k=j$.
Combined, the considerations above imply that
\begin{equation}\label{eq:pivolta_zero}
    \E[|\text{Piv}_{0}|] \leq 3\cdot\E\bigg[\sum_{\ell=1}^{\infty} \sum_{\vv{ij} \in \vv{E_{n}}} \sum_{k: \vv{jk} \in \vv{E_{n}}}\mathbbm{1}_{\omega(\ell)=\vv{ij}}\cdot\mathbbm{1}_{\omega(\ell+1)=\vv{jk}} \cdot\mathbbm{1}_{\eta_{\ell-1}(i) \neq \eta_{\ell-1}(j)}\cdot\mathbbm{1}_{\text{Dict}(\omega^{\ell+1})=k}\bigg].
\end{equation}

Observe now that for any fixed $i,j,k,\ell$, the indicators in \eqref{eq:pivolta_zero} are independent. 
This is due to the fact that $\{\eta_{\ell-1}(i) \neq \eta_{\ell-1}(j)\}$ depends on the edge selections up to $\omega(\ell-1),$ the event $\{\text{Dict}(\omega^{\ell+1})=k\}$ depends on the edge selections from $\omega(\ell+2)$, and the first two events, naturally, only depend on $\omega(\ell)$, and $\omega(\ell+1)$, respectively. 
As discussed in Section \ref{prelim}, for every $\ell$, $\text{Dict}(\omega^{\ell+1})$ is a uniform vertex. Furthermore, $\P(\omega(\ell)=\vv{ij}) = \frac{1}{2|E_{n}|}$. This yields
\begin{equation}\label{expctpiv}
\begin{split}
\E[|\text{Piv}_{0}|]& \leq \frac{3}{n
|E_{n}|^{2}}\cdot\E\bigg[\sum_{\ell=1}^{\infty} \sum_{\vv{ij} \in \vv{E}_{n}} \sum_{k: \vv{jk} \in \vv{E_{n}}}\mathbbm{1}_{\eta_{\ell-1}(i) \neq \eta_{\ell-1}(j)}\bigg]\\
&\leq \frac{3\cdot\Delta_{n}}{n
|E_{n}|^{2}}\cdot\E\bigg[\sum_{\ell=1}^{\infty} \sum_{\vv{ij} \in \vv{E}_{n}} \mathbbm{1}_{\eta_{\ell-1}(i) \neq \eta_{\ell-1}(j)}\bigg].
\end{split}
\end{equation}

To estimate this sum, define for $\ell \geq 0$,
\begin{equation}\label{eq:p}
p_{\ell} :=\frac{1}{2|E_{n}|}\sum_{\vv{ij}\in \vv{E_{n}}} \mathbbm{1}_{\eta_{\ell}(i) \neq \eta_{\ell}(j)},
\end{equation}
so that \eqref{expctpiv} can be rewritten as
\begin{equation}\label{eq:pivotal_bound}
\E[|\text{Piv}_{0}|] \leq \frac{6\cdot\Delta_{n}}{
n |E_{n}|}\cdot\E\bigg[\sum_{\ell=0}^{\infty} p_{\ell}\bigg],
\end{equation}
To conclude the proof, we claim that, for any initial configuration $\eta_{0}$,
\begin{equation}\label{finitestimate}
    \E\bigg[\sum_{\ell=0}^{\infty} p_{\ell}\bigg] \leq \frac{n^{2}}{4},
\end{equation}
This is a consequence of Lemma \ref{rwlemma}. Define the process $X_{\ell}$ to be the number of vertices with opinion $1$ and $p_{\ell}$ as in \eqref{eq:p} above, both processes are adapted to the filtration $\mathcal{F}_{\ell} := \sigma(\omega(i):i\leq\ell)$. 
Moreover, $X_{\ell} \in [0,n],$ and the process $(X_{\ell})_{\ell\in \N}$ is a $(1-p_{\ell})-$lazy discrete-time random walk, since each edge $\vv{ij}$ is picked with the same probability.
Let $\tau$ denote the hitting time of the set $\{0,n\}$ by this process. Observe that if $X_{\ell} \in \{1,...,n-1\},$ then $p_{\ell} \geq \frac{1}{2|E_{n}|}$ while $p_{\ell}=0$ for $\ell > \tau$ . Lemma \ref{rwlemma} implies~\eqref{finitestimate} which, when combined with~\eqref{eq:pivotal_bound} concludes the proof.
\end{proof}

Theorem \ref{stabtheorem} is now a direct consequence of the results above.

\begin{proof}[Proof of Theorem~\ref{stabtheorem} ]
Simply notice that Lemmas~\ref{l_pivotal_bound} and~\ref{lemma_expected_piv} imply
\begin{equation}\label{eq:pivotal_integral2}
\P\big( f_{n}(\omega) \neq f_{n}(\omega_{t}) \big) \leq t \E[|\text{Piv}_{0}|] \leq \frac{3 t n \Delta_{n}}{|E_{n}|},
\end{equation}
concluding the proof.
\end{proof}

\section{Brownian case}\label{sec: Brownian}

Let us start this section by properly introducing the notion of noise that we consider for the continuous-time dynamics.
In order to define the Brownian noise, we consider a decorated Poisson point process $\mathscr{P}$, defined on $\vec{E}_{n} \times \R_{+} \times C[0, +\infty)$ with intensity $\mu \otimes \lambda \otimes W$.
Here, $\mu$ is the counting measure on directed edges, $\lambda$ is the Lebesgue measure on $\R_{+}$, and $W$ denotes the Wiener measure on $C[0, +\infty)$.
Given $(w,t,B)\in \mathscr{P}$ we use the reflected Brownian motion $(|t+B_s|)_{s\geq0}$ as the activation times of the edge $w$.
Whenever we consider a realization $\mathscr{P} = \big(\omega(k), t(k), B(k)\big)_{k \in \N}$, we will always assume that this list is ordered in a way such that the sequence of times $t(k)$ is non-decreasing.

Note that there is a 0-measure set of times at which two edges are activated at the same time. If these edges are disjoint, the resulting opinions do not depend on the order in which the edges are activated. 
This will not have any impact, since this is a 0-measure set of times. 
There are a.s. no times at which three or more edges are activated at once.

Recall that $\theta$ denotes a Poisson point process on $\vec{E}_{n} \times \R_{+}$ with intensity measure $\mu \otimes \lambda$ which will be realized simply as the projection of $\mathscr{P} = (\omega(k), t(k), B(k))_{k \in \N}$ onto its first two coordinates.
Given a noise level $s>0$, define the noisy configuration $\theta_{s}$ as the one obtained from reordering the set $(\omega_{s}(k), |t(k)+B_{s}(k)|)_{k \in \N}$ so that the clock ticks $(|t(k)+B_{s}(k)|)_{k \in \N}$ are in non-decreasing order. This reordering is almost surely uniquely defined, as no two times are the same with probability one. Furthermore, since the Lebesgue measure is invariant with respect to the reflected Brownian motion, it follows that $\theta_{s} \sim \theta$.

For some time $\ell$ to be specified later, let $L$ be the number of time stamps in $[0,\ell]$.
Define the auxiliary sequence of  functions
\begin{equation}
    g_{n}(\xi_{0}, \theta) = f_{n}(\xi_{0},\omega) \mathbbm{1}_{\mathcal{C}(\theta) \leq L(\theta)},
\end{equation}
The advantage of the functions $g_{n}$ is that they depend only on the restriction of the Poisson point process $\theta$ to the interval $[0,\ell]$, a fact that will allow us to make use of counting arguments.
Since our result does not take into account the distribution of $\xi_{0}$, we drop the dependency of $g_{n}$ on these variables, as in the previous section.

In order to reduce our time horizon of the noise to a finite interval, we shall restrict our domains and noise to the time stamps contained in an interval $[0,m]$, where $m \geq \ell$ will be chosen later. More precisely, consider the process $\mathscr{P}$ restricted to $\vec{E}_{n} \times [0,m] \times C[0, +\infty)$.
Denote by $M = M(\theta)$ the number of points of $\theta = \big(\omega(k), t(k)\big)_{k \geq 1}$ such that $t(k) \in [0,m]$.
Given a realization $\theta[m] = \big(\omega(k), t(k)\big)_{k \leq M}$ of this process, we define the noised configuration $\theta_{s}[m]$ by running, for each time stamp $t(k)$, an independent length-$s$ Brownian motion on $[0,m]$, reflected at \emph{both} endpoints.
In other words, $\theta_{s}[m]$ is constructed from $\mathscr{P}$ via
\begin{equation}
    \theta_{s}[m] = (\omega_{s}(k), R_{m}(t(k)+B_{s}(k)))_{k \leq M},
\end{equation}
where, for $x\in \R$,
\begin{equation}\label{eq:reflection}
    R_{m}(x) = \begin{cases}
        x-2zm, & \text{if } x \in [2zm, (2z+1)m), \text{ for some } z \in \Z \\
        (2z+2)m-x, & \text{if } x \in [(2z+1)m, (2z+2)m), \text{ for some } z \in \Z.
    \end{cases}
\end{equation}

We first prove that this modified system is unlikely to differ from the original one in the interval $[0,\ell]$.
We say that $\theta_s=\theta_s[m]$ in $[0,\ell]$ if the set of edges and activation times contained in the interval $[0,\ell]$ is the same.
\begin{lemma}\label{l_modified_truncation}
Fix the graph $G_n$, $\ell>0$, and $s>0$. Then 
\begin{equation}\label{comparisonfg}
    \P \big( \theta_{s} \neq \theta_{s}[m] \text{ in the interval } [0,\ell] \big) \leq \frac{4\sqrt{s}|E_{n}|}{(m-\ell)} \Big( \frac{\sqrt{s}}{m-\ell}+\ell \Big) e^{-\frac{(m-\ell)^{2}}{8s}}.
\end{equation} 
\end{lemma}

\begin{proof}
In order to bound the probability that the two processes differ in the interval $[0,\ell]$, we consider the following auxiliary events
\begin{equation}
    A = \bigg\{
    \begin{array}{c}
    \text{no Brownian motion associated with a particle starting in the interval} \\
    \big[ 0, m \big] \text{ has displacement larger than } \frac{1}{2}(m-\ell) \text{ within time } s 
    \end{array}
    \bigg\}
\end{equation}
and 
\begin{equation}
    B = \bigg\{
    \begin{array}{c}
    \text{ no particle starting at } x \geq m \text{ has displacement} \\ \text{larger than } x-\ell \text{ within time } s
    \end{array}
    \bigg\}.
\end{equation}

Notice that in the event $B$, no particle starting outside the interval $[0,m]$ falls in the interval $[0,\ell]$ at time $s$. Furthermore, in the event $A$, all particles that end up in the interval $[0,\ell]$ at time $s$ start in the interval $\big[0, \frac{1}{2}(m+\ell) \big]$ and do not touch $\ell$ before time $s$.
In particular, in the event $A \cap B$, $\omega_{s} = \omega_{s}[m]$ in the interval $[0,\ell]$.

To conclude the proof, we need to bound the probability of $A^{c}$ and $B^{c}$. First, let us observe that, for $y \geq 0$,
\begin{equation}
    \P \Big( \sup_{t \leq s} |B_t| \geq y \Big) \leq 2 \P \Big( \sup_{t \leq s} B_t \geq y \Big) \leq 2 \P \big( |B_s| \geq y \big) \leq 4 \P \big( B_s \geq y \big) \leq \frac{4\sqrt{s}}{y \sqrt{2\pi}} e^{-\frac{y^{2}}{2s}}. 
\end{equation}

We now bound
\begin{equation}
    \P \big( A^{c} \big) \leq |E_{n}|\cdot \ell \cdot \P \bigg( \sup_{t \leq s} |B_t| \geq \frac{1}{2}(m-\ell) \bigg) \leq \frac{8\sqrt{s}|E_{n}| \ell}{(m-\ell) \sqrt{2\pi}} e^{-\frac{(m-\ell)^{2}}{8s}}.
\end{equation}
As for the event $B$, we have
\begin{equation}
\begin{split}
    \P \big( B^{c} \big) & \leq \int_{m}^{+\infty} |E_{n}| \P \Big(\sup_{t \leq s} |B_t| \geq x-\ell \Big) \, \textnormal{d} x \\
    & \leq \frac{4\sqrt{s}|E_{n}|}{\sqrt{2\pi}} \int_{m}^{\infty} \frac{1}{x-\ell} e^{-\frac{(x-\ell)^{2}}{2s}} \, \textnormal{d} x \\
    & \leq \frac{4s|E_{n}|}{(m-\ell)^{2}\sqrt{2\pi}}e^{-\frac{(m-\ell)^{2}}{2s}}.
    \end{split}
\end{equation}
Combining the last two bounds completes the proof.
\end{proof}

To highlight the fact that we are now restricting ourselves to a finite domain, we will use the notation $h_{n}(\theta[m]):= g_{n}(\theta)$ for the function $h_{n}$ evaluated in the truncated configuration.
When there is no risk of confusion, we will abuse notation and shorten $\theta[m]$ to $\theta$.

Recall $M$ denotes the number of points in the time interval $[0,m]$. For $1 \leq i,j \leq M$, define the event $B_{s}( i \to j)$ where the Brownian motion associated with the time stamp $t(i)$ moves it to the interval $\big(t(j), t(j+1)\big)$. More formally:
\begin{equation}\label{eq:Brownian_move}
    B_{s} (i \to j) = \big\{ R_{m}\big(t(i)+B_{s}(i)\big) \in \big(t(j),t(j+1)\big)\big\}.
\end{equation}
Note that this ignores the change of the other timestamps. This will be justified in due course.

The following lemma bounds the probability that the noise affects $h$ in terms of the pivotal transitions (recall Definition~\ref{def_pivotality}), weighted by the events where the Brownian motions cause such transitions to occur. Note that each term in the sums below corresponds to a single particle being displaced by a Brownian motion while the others stay in place.
\begin{lemma}\label{russolemma}
  \begin{equation}\label{nonsymmetricsum}
    \P \big( h_{n}( \theta) \neq h_{n}(\theta_s) \big) \leq \E\bigg[\sum_{i=1}^{M} \sum_{j=1}^{M} \mathbbm{1}_{f_{n}(\omega) \neq f_{n}(\omega_{i \to j})} \mathbbm{1}_{B_{s} (i \to j)} \mathbbm{1}_{ i \wedge j \leq L} \bigg] +2\E \big[ M \mathbbm{1}_{C(\theta) > L(\theta)-1} \big].
   \end{equation}
\end{lemma}

\begin{proof}
For this proof, we interpolate between the processes $\omega$ and $\omega_{s}$ as follows. Consider a Poisson point process $\mathscr{P}^{*}$ taking values on $\vec{E}_{n} \times \R_{+} \times C[0, +\infty) \times [0,1]$ with intensity $\mu \otimes \lambda \otimes W \otimes U$, where $U$ denotes the uniform measure on $[0,1]$.
Given a realization $(\omega(k), t(k), B(k), U_k)_{k \in \N}$ of such a process, define $\theta^{q}$ as the collection 
\begin{equation}
    \Big(\omega(k), R_{m}\big(t(k)+B_s(k) \mathbbm{1}_{[0,q]}(U_k) \big)\Big)_{k \in \N},
\end{equation}
combined with a reordering so that the time stamps are in non-decreasing order.
In $\theta^{q}$, the only particles moved are those such that $u\leq q$. In particular, one has $\theta^{0} = \theta$ and $\theta^{1} = \theta_{s}$.

Before proceeding, let us argue that the distribution of $\theta^{q}$ does not depend on the value of $q$.
In fact, notice that the collections $\{ (\omega(k), R_{m}(t(k)+B_s(k))) : k \in \Z \text{ and } U_k \leq q \}$ and $\{ (\omega(k), t(k)) : k \in \Z \text{ and } U_k > q \}$ are two independent Poisson point processes taking values on $\vec{E} \times \R$ with intensity measures $q \mu \otimes \lambda$ and $(1-q) \mu \otimes \lambda$, respectively.
The process $\theta^{q}$ can be obtained as the union of the two collections and is therefore distributed as a Poisson point process on $\vec{E}_{n} \times \R$ with intensity $\mu \otimes \lambda$. An argument analogous to this one implies that the distribution of a pair $(\theta^{q}, \theta^{q+\tilde{q}})$ depends only on the difference $\tilde{q}$ and not on the value of $q$.

Observe now that the fundamental theorem of calculus yields
\begin{equation}\label{eq:lemma_integral}
\P \big( h_{n}(\theta) \neq h_{n}(\theta_{s}) \big) = \int_{0}^{1} \frac{\partial}{\partial q} \P \big( h_{n}(\theta^{0}) \neq h_{n}(\theta^{q}) \big) dq.
\end{equation}

Calculating the derivative above, we obtain
\begin{equation}\label{eq:lemma_deriative}
\begin{split}
\frac{\partial}{\partial q} \P \big( h_{n}(\theta) & \neq h_{n}(\theta^{q}) \big) = \lim_{ \eps \to 0} \frac{\P \big( h_{n}(\theta) \neq h_{n}(\theta^{q+\eps}) \big)-\P \big( h_{n}(\theta) \neq h_{n}(\theta^{q}) \big)}{\eps} \\
& \leq \lim_{\eps \to 0} \frac{\P \big( h_{n}(\theta) \neq h_{n}(\theta^{q}) \big) + \P\big( h_{n}(\theta^{q}) \neq h_{n}(\theta^{q+\eps}) \big)-\P \big( h_{n}(\theta) \neq h_{n}(\theta^{q}) \big)}{\eps} \\
& = \lim_{\eps \to 0} \frac{ \P \big( h_{n}(\theta) \neq h_{n}(\theta^{\eps}) \big)}{\eps},
\end{split}
\end{equation}
where the third equality follows from the fact that the distributions of the pairs $(\theta^{q}, \theta^{q+\varepsilon})$ and $(\theta, \theta^{\varepsilon})$ are the same.

Analogously to the derivation of the Margulis-Russo formula, if we denote by $\theta(i)$ the configuration when only the $i$-th time entry is noised, we have 
\begin{equation}\label{eq:lemma_margulis_russo}
  \begin{split}
    \P \big( h_{n}(\theta) \neq h_{n}(\theta^{\eps}) \big) & \leq \E\bigg[\sum_{i=1}^{M} \mathbbm{1}_{h_{n}(\theta) \neq h_{n}(\theta(i))} \mathbbm{1}_{U_{i} \leq \eps} +\sum_{\substack{ i_{1},i_{2} \leq M \\ i_{1} \neq i_{2}}} \mathbbm{1}_{U_{i_{1} \leq \eps}}\mathbbm{1}_{U_{i_{2}\leq \eps}}\bigg] \\
    & \leq \varepsilon \E\bigg[\sum_{i=1}^{M} \mathbbm{1}_{h_{n}(\theta) \neq h_{n}(\theta(i))} \bigg] + \E[M^{2}]\eps^{2}.
  \end{split}
\end{equation}
In particular, combining Equations~\eqref{eq:lemma_integral},~\eqref{eq:lemma_deriative}, and~\eqref{eq:lemma_margulis_russo}, we obtain
\begin{equation}
    \P \big( h_{n}( \theta) \neq h_{n}(\theta_s) \big) \leq \E\bigg[\sum_{i=1}^{M} \mathbbm{1}_{h_{n}(\theta) \neq h_{n}(\theta(i))} \bigg].
\end{equation}

Let us now bound the last term in the equation above. Recall that $h_{n}(\theta) =  f_{n}(\omega) \mathbbm{1}_{C(\theta) \leq L(\theta)}$. In particular, in order for $h_{n}(\theta) \neq h_{n}(\theta(i))$, either $f_{n}(\omega) \neq f_{n}(\omega_{i \to j})$ and the Brownian motion associated to the $i$-th time stamp moves to an interval $(t_{j}, t_{j+1})$, or $C(\theta) > L(\theta)$, or $C(\theta(i)) > L(\theta(i))$. Furthermore, we can further require that either $i \leq L$ or $j \leq L$, since otherwise the movement of the $i$-th time stamp to the interval $(t_{j}, t_{j+1})$ does not alter the value of $h_{n}$. This yields the bound
\begin{equation}\label{eq:lemma_splititng_h}
    \sum_{i=1}^{M} \mathbbm{1}_{h_{n}(\theta) \neq h_{n}(\theta(i))} \leq \sum_{i=1}^{M} \bigg( \sum_{j=1}^{M} \mathbbm{1}_{f_{n}(\omega) \neq f_{n}(\omega_{i \to j})} \mathbbm{1}_{B_{s} (i \to j)} \mathbbm{1}_{ i \wedge j \leq L} \bigg)  + \mathbbm{1}_{C(\theta) > L(\theta)} + \mathbbm{1}_{C(\theta(i)) > L(\theta(i))}.
\end{equation}

Observe now that the distribution of $C(\theta(i))$ is the same as the one from $C(\theta)$, since it depends only on the list of edge selections $\omega$ and $\omega(i)$. Furthermore, we have $L(\theta(i)) \geq L(\theta)-1$, since at most one time-stamp is removed from the interval $[0, \ell]$. In particular, we obtain
\begin{equation}
    \E\bigg[\sum_{i=1}^{M} \mathbbm{1}_{C(\theta(i)) > L(\theta(i))} \bigg] \leq \E \big[ M \mathbbm{1}_{C(\theta) > L(\theta)-1} \big],
\end{equation}
which, when combined with~\eqref{eq:lemma_splititng_h} yields
\begin{equation}
     \E\bigg[\sum_{i=1}^{M} \mathbbm{1}_{h_{n}(\theta) \neq h_{n}(\theta(i))} \bigg] \leq \E\bigg[\sum_{i=1}^{M} \sum_{j=1}^{M} \mathbbm{1}_{f_{n}(\omega) \neq f_{n}(\omega_{i \to j})} \mathbbm{1}_{B_{s} (i \to j)} \mathbbm{1}_{ i \wedge j \leq L} \bigg] +2\E \big[ M \mathbbm{1}_{C(\theta) > L(\theta)-1} \big],
\end{equation}
concluding the proof.
\end{proof}

\nc{c_coal_2}
Our next goal is to upper bound the right-hand side of \eqref{nonsymmetricsum}. 
We start with the second term, which is bounded with the aid of the following lemma.
Recall that $L(\theta)$ is the number of edge activations in $[0,\ell]$, and similarly $M$ and $m$.

\begin{lemma}\label{lemma:coal}
There exists $\uc{c_coal_2}>0$ such that
\begin{equation}
    \E \big[ M \mathbbm{1}_{C(\theta) \ge L(\theta)} \big] \leq  12m |E_{n}|e^{-\uc{c_coal_2}\frac{\ell}{n|E_{n}|}}.
\end{equation}
\end{lemma}

\begin{proof}
Start by noting that the Cauchy-Schwarz inequality yields
\begin{equation}\label{Mcauchy}
  \begin{split}
    \E \big[ M \mathbbm{1}_{C(\theta) \ge L(\theta)} \big] & \leq \sqrt{\E[M^{2}]\P(\mathcal{C}(\theta)\ge L(\theta))} \\
    & \leq 12 m|E_{n}| \sqrt{\P(\mathcal{C}(\theta)\ge L(\theta))}.
  \end{split}
\end{equation}

We now go back to continuous time to employ Theorem~\ref{t:coal}. Notice that
\begin{equation}
    \P\big(\mathcal{C}(\theta) \geq L(\theta)\big) = \P(\mathcal{C}^{*} \geq \ell ).
\end{equation}
Markov inequality combined with Theorem~\ref{t:coal} directly implies, for every $t \geq 0$,
\begin{equation}
    \P(\mathcal{C}^{*} \geq t ) \leq \frac{\uc{c_coal} T_{\textnormal{hit}}}{t}.
\end{equation}

Now, to examine coalescence up to time $kt$, we divide the time interval into $k$ disjoint intervals of size $t$. If the random walks do not coalesce in time $kt$, then there cannot be coalescence in any of these sub-intervals. Independence then yields
\begin{equation}
    \P(\mathcal{C}^{*} \geq kt ) \leq \P(\mathcal{C}^{*} \geq t )^{k} \leq \bigg( \frac{\uc{c_coal} T_{\textnormal{hit}}}{t} \bigg)^{k}.
\end{equation}
We now choose $t =  2e\uc{c_coal}(n-1)|E_{n}|$ so that
\begin{equation}
    \P(\mathcal{C}^{*} \geq kt ) \leq \P(\mathcal{C}^{*} \geq t )^{k} \leq e^{-k}.
\end{equation}
We now apply this to $k = \lfloor \frac{\ell}{t} \rfloor$ to obtain
\begin{equation}
    \P(\mathcal{C}^{*} \geq \ell ) \leq \P(\mathcal{C}^{*} \geq \ell )^{k} \leq e^{-\lfloor \frac{\ell}{t} \rfloor} \leq e^{- \frac{\ell}{2e\uc{c_coal}(n-1)|E_{n}|}-1},
\end{equation}
which yields, when combined with Equation \eqref{Mcauchy},
\begin{equation}
    \E \big[ M \mathbbm{1}_{C(\theta) \geq L(\theta)} \big] \leq  12 m |E_{n}| e^{-\uc{c_coal_2}\frac{\ell}{n|E_{n}|}},
\end{equation}
for $\uc{c_coal_2} = \frac{1}{4e\uc{c_coal}}$, concluding the proof.
\end{proof}

\nc{c_pivotal_bound}
As for the first term in the right-hand side of~\eqref{nonsymmetricsum}, the following lemma provides a bound.
\begin{lemma}\label{pivotalbound}
There exists a constant $\uc{c_pivotal_bound}>0$ such that
  \begin{equation}\label{auxbound}
    \begin{split}
      \E\bigg[\sum_{i=1}^{M} \sum_{j=1}^{M} & \mathbbm{1}_{f_{n}(\omega) \neq f_{n}(\omega_{i \to j})} \mathbbm{1}_{B_{s} (i \to j)} \mathbbm{1}_{ i \wedge j \leq L} \bigg] \\
      & \leq \uc{c_pivotal_bound} \sqrt{s} \bigg( |E_{n}| (1+ m|E_{n}|)  \frac{m}{m-\ell}e^{-(m-\ell)^{2}/2s} + n \Delta_{n} \big( 1 + e^{2|E_{n}|^{2}s} + |E_{n}| \sqrt{s} \big) \bigg).
    \end{split}
  \end{equation}
\end{lemma}

The analysis of this lemma will be done in two steps. We first analyze the terms referring to the Brownian motions, and then the amount of pivotal swaps will be estimated with the help of the results in Section~\ref{stabsec}.

\begin{proof}
Denote by $\mathcal{F}_{\text{init}}$ the $\sigma$-algebra generated by the time stamps  $\{t(k)\}_{k \geq 1}$ of the random vector $\theta$. This allows us to write
\begin{equation}
   \begin{split}
      \E\bigg[\sum_{i=1}^{M} \sum_{j=1}^{M} & \mathbbm{1}_{f_{n}(\omega) \neq f_{n}(\omega_{i \to j})} \mathbbm{1}_{B_{s} (i \to j)} \mathbbm{1}_{ i \wedge j \leq L} \bigg] \\
      & \qquad = \E\bigg[\sum_{i =1}^{M} \sum_{j=1}^{M} \P\big( f_{n}(\omega) \neq f_{n}(\omega_{i \to j})\big)  \P\big( B_{s} (i \to j) \big| \mathcal{F}_{\text{init}} \big)\mathbbm{1}_{ i \wedge j \leq L}  \bigg].
   \end{split}
\end{equation}

Let us assume $i < j$, so that the first estimate we perform will be on the sum
\begin{equation}
      \E\bigg[\sum_{i =1}^{L} \sum_{j=i+1}^{M} \P\big( f_{n}(\omega) \neq f_{n}(\omega_{i \to j})\big)  \P\big( B_{s} (i \to j) \big| \mathcal{F}_{\text{init}} \big)\bigg],
\end{equation}
the sum over $j < i$ is treated analogously.

We begin by examining $\P \big( B_{s} (i \to j) \big\vert \mathcal{F}_{\text{init}} \big)$. Observe that, if $R_{m}(t(i) + B_{i}(s)) \in  (t(j), t(j+1))$, then either
\begin{equation}\label{swapwalk}
    t(i)+B_{s}(i) \in (-t(j+1),-t(j)) \cup (t(j),t(j+1))
    \quad \text{or} \quad
    |t(i)+B_{s}(i)| > m.
\end{equation}
Since $t(i) \leq \ell$, the last condition above implies $|B_{s}(i)| \geq m-\ell$, which in turns yields the following bound
\begin{equation}\label{brownianest}
\P\big( B_{s} (i \to j) \big| \F_{\text{init}} \big) \leq \int_{t(j) - t(i)}^{t(j+1)-t(i)}P_{s}(y)dy + \int_{t(i)+ t(j)}^{t(i)+t(j+1)} P_{s}(y)dy+2\int_{m-\ell}^{\infty} P_{s}(y)dy,
\end{equation}
where $P_{s}(y) := \frac{1}{\sqrt{2\pi s}} e^{-y^{2}/2s}$ denotes the density of the normal distribution. Since $ t(i)+t(j) \geq t(j)-t(i),$ for every  $i,j \geq 1$,
\begin{equation}\label{leftmirror}
\P\big( B_{s} (i \to j) \big| \F_{\text{init}} \big) \leq 2 \int_{t(j) - t(i)}^{t(j+1)-t(i)}P_{s}(y)dy + 2\int_{m-\ell}^{\infty} P_{s}(y)dy.
\end{equation}

Observe now that
\begin{equation}\label{lrpiv}
\P(f_{n}(\omega) \neq f_{n}(\omega_{i\to j})) \leq \sum_{k=i}^{j-1}\P(f_{n}(\omega_{i\to k}) \neq f_{n}(\omega_{i\to k+1})).
\end{equation}
Combining Equations~\eqref{leftmirror} and~\eqref{lrpiv} we get
\begin{equation}\label{furtherbound}
\begin{split}
  \E\bigg[ & \sum_{i =1}^{L} \sum_{j=i+1}^{M} \P\big( f_{n}(\omega) \neq f_{n}(\omega_{i \to j}) \big) \P(B_{s}\big(i \rightarrow j \big) | \mathcal{F}_{\text{init} } \big) \bigg] \\ 
  & \leq 2 \E\bigg[ \bigg(\sum_{i=1}^{L} \sum_{j=i+1}^{M} \sum_{k=i}^{j-1} \P\big(f_{n}(\omega_{i\to k}) \neq f_{n}(\omega_{i\to k+1})\big) \int_{t(j) - t(i)}^{t(j+1)-t(i)}P_{s}(y)dy\bigg) + M^{2}\int_{m-\ell}^{\infty} P_{s}(y)dy \bigg].
\end{split}
\end{equation}
We will bound the two terms on the right-hand side of the equation above separately.

The second term above is simpler, and follows by noticing that
\begin{equation}\label{rightmirror}
  \begin{split}
    \E[M^{2}] \int_{m-\ell}^{\infty} P_{s}(y) dy & \leq 2m|E_{n}| (1+ 2m|E_{n}|)  \frac{\sqrt{s}}{\sqrt{2\pi}(m-\ell)}e^{-(m-\ell)^{2}/2s} \\
    & \leq |E_{n}|\sqrt{s} (1+ 2m|E_{n}|)  \frac{m}{m-\ell}e^{-(m-\ell)^{2}/2s}.
  \end{split}
\end{equation}

To bound the first term in the right-hand side of~\eqref{furtherbound}, start by noticing that stationarity implies that, for any $i, k \geq 1$,
\begin{equation}
    \P(f_{n}(\omega_{i \to k}) \neq f_{n}(\omega_{i \to k+1})) = \P(f_{n}(\omega) \neq f_{n}(\omega_{k \to k+1})),
\end{equation}
so that
\begin{equation}\label{eq:lemma_bound_sums}
  \begin{split}
    \E\bigg[ \sum_{i=1}^{L} \sum_{j=i+1}^{M} & \sum_{k=i}^{j-1} \P\big(f_{n}(\omega_{i\to k}) \neq f_{n}(\omega_{i\to k+1})\big) \int_{t(j) - t(i)}^{t(j+1)-t(i)}P_{s}(y)dy \bigg] \\
    & \leq \E\Bigg[ \sum_{k=1}^{M-1} \P(f_{n}(\omega)\neq f_{n}(\omega_{k \to k+1})) \sum_{i=1}^{L} \sum_{j=i+1}^{M} \mathbbm{1}_{\{i \leq k \leq j-1\}} \int_{t(j) - t(i)}^{t(j+1)-t(i)} P_{s}(y) dy \Bigg] \\
    & \leq \E\Bigg[ \sum_{k=1}^{M-1} \P(f_{n}(\omega) \neq f_{n}(\omega_{k\to k+1})) \sum_{i=1}^{L} \mathbbm{1}_{i \leq k} \sum_{j= k+1}^{M} \int_{t(j) - t(i)}^{t(j+1)-t(i)} P_{s}(y) dy \Bigg] \\
    & \leq \sum_{k=1}^{\infty} \P(f_{n}(\omega) \neq f_{n}(\omega_{k \to k+1})) \sum_{i=1}^{k} \E\bigg[ \int_{t(k+1) - t(i)}^{\infty} P_{s}(y) dy \bigg].
  \end{split}
\end{equation}
Since $t(j+1) - t(j) \sim \text{Exp}(2|E_{n}|)$, for every $j \geq 1$, if we denote by $(X_{j}, \gamma_{j})$ a pair of independent random variables with distributions $N(0,s)$ and $\Gamma(j, 2|E_{n}|)$, respectively,
we can bound the right-hand side of Equation~\eqref{eq:lemma_bound_sums} by\footnote{Recall the sum of $n$ independent random variables with distribution $\text{Exp}(\lambda)$ has distribution $\Gamma (n, \lambda)$.}
\begin{equation}\label{eq:lemma_nice_bound}
    \sum_{k=1}^{\infty} \P(f_{n}(\omega) \neq f_{n}(\omega_{k \to k+1})) \sum_{j=1}^{\infty} \P( X_{j} \geq \gamma_{j}).
\end{equation}

Now, Lemma~\ref{lemma_expected_piv} yields
\begin{equation}\label{fpivotalbound}
\sum_{k=1}^{\infty}\P(f_{n}(\omega)\neq f_{n}(\omega_{k \to k+1})) \leq 3 \frac{n \Delta_{n}}{|E_{n}|}.
\end{equation}

\nc{c_lemma_1}

As for the second summation in~\eqref{eq:lemma_nice_bound}, we first consider the case $j=1$ separately, and bound the probability of interest considering the cases where $\gamma_{1} \leq \sqrt{s}$ or $\gamma_{1}>\sqrt{s}$. We obtain
\begin{equation}\label{normalone}
  \begin{split}
     \P(X_{1} \geq \gamma_{1}) & \leq \P\big( \gamma_{1} \leq \sqrt{s} \big) + \P\big( N(0,s) \geq \gamma_{1} >  \sqrt{s} \big) \\
     & = (1-e^{-2|E_{n}|\sqrt{s}}) + \frac{1}{\sqrt{2\pi}} \int_{\sqrt{s}}^{\infty}\bigg(\int_{x/\sqrt{s}}^{\infty}e^{-y^{2}/2}dy\bigg) 2|E_{n}| e^{-2|E_{n}|x}dx \\
     & \leq 2|E_{n}|\sqrt{s} \bigg( 1 + \frac{1}{\sqrt{2\pi}}\int_{\sqrt{s}}^{\infty} \frac{1}{x}e^{-x^{2}/2s}  e^{-2|E_{n}|x}dx \bigg) \\
     & \leq 2 |E_{n}| \sqrt{s}\bigg( 1 +  \frac{e^{2|E_{n}|^{2}s}}{\sqrt{2\pi}} \int_{\sqrt{s}}^{\infty}\frac{1}{x}e^{-(x+2|E_{n}|s)^{2}/2s} dx \bigg) \\
     & \leq 2|E_{n}|\sqrt{s} \bigg(1 + \frac{e^{2|E_{n}|^{2}s}}{\sqrt{2\pi}} \int_{1}^{\infty} \frac{1}{y}e^{- (y+2|E_{n}|\sqrt{s})^{2}/2} dy \bigg) \\
     & \leq 2|E_{n}|\sqrt{s} \big( 1 + e^{2|E_{n}|^{2}s} \big).
  \end{split}
\end{equation}
When $j \geq 1$, notice that Chernoff's bound yields
\begin{equation}\label{gammabound}
  \begin{split}
     \P(\gamma_{j} \leq \sqrt{s}\log j) & = \inf_{t>0} e^{t\sqrt{s}\log j}  \bigg( 1 + \frac{t}{2|E_{n}|}\bigg)^{-j} \\
     & \leq e \bigg( 1 + \frac{1}{2 |E_{n}| \sqrt{s} \log j}\bigg)^{-j},
  \end{split}
\end{equation}
where in the last inequality we choose $t = (\sqrt{s} \log j )^{-1}$.
Observe that for $j \geq 2$,
\begin{equation}
  \begin{split}
     \bigg( 1 + \frac{1}{2 |E_{n}| \sqrt{s} \log j}\bigg)^{j} & \geq 1 + j \frac{1}{2 |E_{n}| \sqrt{s} \log j} + \frac{j(j-1)}{2} \bigg( \frac{1}{2 |E_{n}| \sqrt{s} \log j } \bigg)^{2} \\
     & \geq \frac{1}{16 |E_{n}|^{2} s} \frac{j^{2}}{\log^{2} j}.
  \end{split}
\end{equation}
Hence, summing over $j \geq 2$,
\begin{equation}
\sum_{j=2}^{\infty}\P(\gamma_{j} \leq \sqrt{s} \log j) \leq  16 e |E_{n}|^{2} s \sum_{j=2}^{\infty} \frac{\log^{2} j}{j^{2}} \leq  100|E_{n}|^{2}s .
\end{equation}

We now bound
\begin{equation}\label{normalgamma}
  \begin{split}
     \sum_{j=2}^{\infty} \P( X_{j} \geq \gamma_{j}) & \leq 100|E_{n}|^{2}s + \sum_{j=2}^{\infty} \E \bigg[\frac{\sqrt{s}}{\sqrt{2\pi} \gamma_{j}} e^{-\gamma_{j}^{2}/2s} \mathbbm{1}_{\gamma_{j} > \sqrt{s} \log j}\bigg]  \\
     & \leq 100 |E_{n}|^{2} s + \sum_{j=2}^{\infty} \frac{\sqrt{s}}{\sqrt{2\pi}} e^{-\frac{1}{2}\log^{2} j} \E\bigg[\frac{1}{\gamma_{j}}\bigg] \\
     & \leq 100 |E_{n}|^{2} s + \frac{2}{\sqrt{2\pi}} |E_{n}| \sqrt{s} \sum_{j=2}^{\infty}\frac{1}{j-1} e^{-\frac{1}{2}\log^{2} j} \\ 
     &\leq 100 |E_{n}|^{2}s + |E_{n}|\sqrt{s}.
  \end{split}
\end{equation}

In particular, Equations~\eqref{normalone} and \eqref{normalgamma} imply that, for any $k\geq 1$,
\begin{equation}
\sum_{i=1}^{k} \E\bigg[ \int_{t(k+1) - t(i)}^{\infty} P_{s}(y) dy \bigg] \leq |E_{n}|\sqrt{s}\Big(3 + 2e^{2|E_{n}|^{2}s}+ 100 |E_{n}|\sqrt{s}\Big).
\end{equation}
This implies with Equation~\eqref{fpivotalbound} that
\begin{equation}
  \begin{split}
     \E\bigg[ \sum_{i=1}^{L} \sum_{j=i+1}^{M} \sum_{k=i}^{j-1} & \P\big(f_{n}(\omega_{i\to k}) \neq f_{n}(\omega_{i\to k+1})\big) \int_{t(j) - t(i)}^{t(j+1)-t(i)}P_{s}(y)dy \bigg] \\
     & \leq 3 n \Delta_{n} \sqrt{s} \bigg( 3 + 2e^{2|E_{n}|^{2}s} + 100 |E_{n}| \sqrt{s} \bigg).
  \end{split}
\end{equation}

The proof is completed by combining the bound above with~\eqref{rightmirror} and adjusting the constants.
\end{proof} 

We are finally ready to prove Theorem~\ref{t_brownian_noise}. Borrowing from the intuition given by percolation crossings, sensitivity is detected up to the threshold where pivotal swaps happen with non-vanishing probability. In Section~\ref{stabsec}, the local structure of the noise is simpler and pivotality is easier to define. However, for the Brownian perturbation, this is less straightforward, and long-range swaps need to be considered. This is the reason why we introduce the events in~\eqref{eq:Brownian_move}. Lemma~\ref{russolemma} is the key step in our approach and establishes an upper bound, which requires the introduction of $\theta^{q}$. In fact, we obtain an interpolation between the edge selections $\theta$ and $\theta_{s}$ that, as $q$ varies, randomly displaces one particle at a time. Finally, this allows us to analyze the behavior of each Brownian move separately, by relying on simple quantitative bounds for the Brownian heat kernel.

\begin{proof}[Proof of Theorem~\ref{t_brownian_noise}]
Combining Lemmas~\ref{l_modified_truncation},~\ref{russolemma}, and~\ref{pivotalbound} we get
\begin{equation}\label{eq_t_big_bound}
   \begin{split}
      \P\big( f_{n}(\theta) \neq & f_{n}(\theta_{s}) \big) \leq \frac{2\sqrt{s}|E_{n}|}{(m-\ell)} \Big( \frac{\sqrt{s}}{m-\ell}+\ell \Big) e^{-\frac{(m-\ell)^{2}}{2s}} + 24m |E_{n}|e^{-\uc{c_coal_2}\frac{\ell}{n|E_{n}|}} \\
      + \uc{c_pivotal_bound} \sqrt{s} & \bigg( |E_{n}| (1+ m|E_{n}|)  \frac{m}{m-\ell}e^{-(m-\ell)^{2}/2s} + n \Delta_{n} \big( 1 + e^{2|E_{n}|^{2}s} + |E_{n}| \sqrt{s} \big) \bigg).
   \end{split}
\end{equation}

Since $G_{n}$ is connected we must have $\frac{n}{2} \leq |E_{n}| \leq n^{2}$. Together with the condition $|E_{n}| \sqrt{s} \leq \uc{c_brownian_1}$, we obtain
\begin{equation}
    n \leq \frac{2\uc{c_brownian_1}}{\sqrt{s}} \quad \text{and} \quad s \leq \uc{c_brownian_1}^{2}.
\end{equation}

Choose now $\ell = n^{3}/s$ and $m = \ell(1 + 1/s) = \frac{\ell}{s}(1+s)$. Substituting these values in the right-hand-side of~\eqref{eq_t_big_bound} and using the bound $|E_{n}| \leq n \Delta_{n}$, we obtain
\begin{equation}
   \begin{split}
      \P\big( f_{n}(\theta) \neq f_{n}(\theta_{s}) \big) & \leq 2 s^{3/2} |E_{n}| \bigg( 1+\frac{s^{7/2}}{n^{6}}\bigg) e^{-\frac{n^{6}}{2s^{5}}} + 24\frac{n^{3}}{s^{2}}(1+\uc{c_brownian_1}^{2}) |E_{n}|e^{-\uc{c_coal_2}\frac{1}{s}} \\
      + \uc{c_pivotal_bound} \sqrt{s} & \bigg( |E_{n}| \Big( 1 + 4\frac{n^{5}}{s^{2}} \big( 1+\uc{c_brownian_1}^{2} \big) \Big) \big( 1 + \uc{c_brownian_1}^{2} \big) e^{-n^{6}/2s^{5}} + n \Delta_{n} \big( 1 + e^{2\uc{c_brownian_1}^{2}} + \uc{c_brownian_1} \big) \bigg) \\
      & \leq n \Delta_{n} \sqrt{s} \bigg( 2 \uc{c_brownian_1}^{2} \big( 1 + \uc{c_brownian_1}^{7} \big) + 24 \frac{8 \uc{c_brownian_1}^{3}}{s^{4}} \big( 1+\uc{c_brownian_1}^{2} \big) e^{-\uc{c_coal_2}\frac{1}{s}} \\
      & \qquad + \uc{c_pivotal_bound} \Big( 1+ 4\frac{n^{5}}{s^{2}} \big(1+\uc{c_brownian_1}^{2} \big) \Big) \big( 1 + \uc{c_brownian_1}^{2} \big) e^{-n^{6}/2s^{5}} + 1 + e^{2\uc{c_brownian_1}^{2}} + \uc{c_brownian_1} \bigg) \\
      & \leq \uc{c_brownian_2}n \Delta_{n}\sqrt{s},
   \end{split}
\end{equation}
where the last bound is obtained by noticing that all the terms in parentheses are bounded in the domain $n \geq 1$ and $s \in (0, \uc{c_brownian_1}^{2})$. This concludes the proof of the theorem. 
\end{proof}

\paragraph{Acknowledgements.}

The authors thank Daniel Ahlberg and Hugo Vanneuville for fruitful discussions during the elaboration of this paper.\par 
The research of GA counted on the support of the Israel Science Foundation grant \# 957/20.
OA and DdlR were supported in part by NSERC.
The research of RB was funded by ``Conselho Nacional de Desenvolvimento Científico e Tecnológico – CNPq'' grants ``Produtividade em Pesquisa'' (308018/2022-2) and ``Projeto Universal'' (402952/2023-5).

\printbibliography

\end{document}